%% file: arXiv_v1.tex
\documentclass[11pt]{amsart}

\input{running_preamble}

\begin{document}


\title[Motion groups, Legendrian links, non-trivial loops of embeddings]{Motion groups, Legendrian links, and non-trivial loops of embeddings}
\author{Gabriel Corrigan and Isacco Nonino}
\address{School of Mathematics and Statistics, University of Glasgow, University Place, Glasgow G12 8QQ}
\thanks{This work was developed and written without the use of artificial intelligence models. The authors do not consent to the processing of this document by artificial intelligence models.}

\begin{abstract}
We introduce, for an embedded Legendrian link~$L$ in a contact 3-manifold~$(M,\xi)$, the group of its \emph{Legendrian motions}. Analogously with the smooth case, we prove that this Legendrian motion group is isomorphic to the fundamental group of the space of unparametrised Legendrian embeddings. For all Legendrian knots in the standard contact 3-sphere, we characterise Legendrian motions in terms of contact mapping classes of the knot exterior. In an appendix, we give a short exposition of our techniques applied to the smooth case, comparing to the classical Dahm homomorphism. 
\end{abstract}

\maketitle

\tableofcontents

\section{Introduction}\label{sec:intro}

Understanding the space of Legendrian embeddings of knots in contact 3-manifolds is one of the central topics in recent low-dimensional contact topology. Recent work of Fern\'andez, Mart\'inez-Aguinaga, Min, and Presas \cite{Min-Fernandez, Fernandez-Aguinaga-Presas, Fernandez-Presas-Aguinaga-Homotopy-type, Aguinaga,Xabi2026} has dealt with the homotopy theory of several spaces of Legendrian embeddings. Despite much progress, even the space of Legendrian embeddings in the standard tight sphere~$(S^3,\xi_{std})$ is still shrouded in mystery. Forgetting the contact structure gives an obvious inclusion 
\[\tag{$\ast$}\label{eqn:inclusion of LEmb into Emb}\lemb\!\left(S^1, (S^3,\xi_{std})\right) \hookrightarrow \emb\!\left(S^1,S^3\right)\]
of the space of Legendrian embeddings into the space of smooth embeddings. The induced map at the~$\pi_0$-level has been studied for many connected components of the space of Legendrian embeddings. It has been shown by Fuchs--Tabachnikov that the map is surjective \cite{Fuchs_Tabachnikov} -- i.e. every smooth isotopy type of smooth knots has a Legendrian representative. On the other hand, these~$\pi_0$-level maps fail to be injective in general: there are families of pairs of Legendrian embeddings which are smoothly isotopic but not \emph{Legendrian isotopic}, that is, isotopic through Legendrian embeddings (e.g. see Eliashberg--Fraser \cite{Eliashberg-Fraser}). As for the~$\pi_1$-level, work of K\'alm\'an \cite{Kalman} produces large families of Legendrian knot types~$L$ which witness failure of injectivity of the homomorphism on~$\pi_1$ induced by the inclusion (\ref{eqn:inclusion of LEmb into Emb}). Higher homotopy groups have been investigated by Mart\'inez-Aguinaga \cite{Aguinaga,Xabi2026}.

Although the correspondence between the smooth and contact categories has been analysed extensively for knot embeddings, the notion of \emph{motions} has not yet been transferred to the contact world. The purpose of this article is to begin this translation for Legendrian links in contact 3-manifolds.

Roughly speaking, a (smooth) motion of a compact submanifold~$N \subset M$ is a path of diffeomorphisms of~$M$ starting at the identity~$\id_M$ and ending at a diffeomorphism which preserves~$N$ setwise. The group of all such is the \emph{motion group}~$\motion{N}{M}$, first defined in this way (topologically) by Dahm \cite{Dahm62} (see Goldsmith's exposition \cite{GoldsmithMotionGroups81}). Wattenberg showed that in fact~$\motion{N}{M}$ is isomorphic to the fundamental group of (a connected component of) the space of \emph{unparametrised embeddings}~$N \hookrightarrow M$ \cite{WattenbergSmoothMotionsUnlink72}.

Our starting contribution is to define a \emph{Legendrian motion group} whose elements are \emph{Legendrian motions}, in direct analogy with Dahm's original definition. We then realise  the Legendrian motion group as the fundamental group of an embedding space. Let~$L \colon \bigsqcup S^1 \hookrightarrow M$ be a Legendrian link in a contact 3-manifold~$(M, \xi)$, and let~$\mathcal{LE}(L)$ be the connected component of 
\[\lsub\!\left(\bigsqcup S^1, (M, \xi)\right) \coloneqq \faktor{\lemb\!\left(\bigsqcup S^1, (M, \xi)\right)}{\diff\!\left(\bigsqcup S^1\right)}\]
containing~$L$.

\begin{customthm}{A}[\ref{thm:legendrian motion group as pi_1 of embedding space}]\label{customthm:legendrian motion group as pi_1 of embedding space}
    The Legendrian motion group~$\lmotion{L}{(M, \xi)}$ is isomorphic to~$\pi_1\!\left(\mathcal{LE}(L)\right)$.
\end{customthm}

Our aim in introducing Legendrian motions is to provide a new language, which may bring new techniques to bear, for investigating Legendrian embeddings and for comparing the smooth and contact categories in this setting. We show that some elementary properties of (smooth) motion groups also hold in the Legendrian setting (see \S\ref{subsec:basic properties of Legendrian motion groups}). The following consequence can be viewed as a potential strategy for obstructing Legendrian isotopy between links in contact 3-manifolds.

\begin{customcor}{B}[\ref{corollary:legendrian_isotopy_implies_isomorphism}]\label{customcor:legendrian isotopy implies isomorphism}
    If~$K, K'$ are two Legendrian isotopic links in~$(M, \xi)$, then their Legendrian motion groups are isomorphic.
\end{customcor}

Along the way to proving \emph{Theorem \ref{customthm:legendrian motion group as pi_1 of embedding space}}, we prove contact analogues of some useful classical results from the smooth setting, which may be of independent interest. In particular, we point out \emph{Theorems \ref{thm:parametrised_fiber_bundle} and \ref{thm:unparametrised_fiber_bundle}}, which give locally trivial fibrations $\cont_c\!\left(M, \xi\right) \to \lemb\!\left(K, (M, \xi)\right)$ and $\cont_c\!\left(M, \xi\right) \to \lsub\!\left(K, (M, \xi)\right)$ respectively, and identify their fibres with special subgroups of $\cont_c\!\left(M, \xi\right)$.

As a first example of how this language could lead to new techniques, we connect Legendrian motions to contact mapping classes. The \emph{Dahm homomorphism} is a natural action of the (smooth) motion group~$\motion{N}{M}$ on~$\pi_1\!\left(M \setminus N\right)$ by (outer, if~$M$ is compact) automorphisms. In the Legendrian setting there does not seem to be a similar natural realisation of motions as automorphisms of the knot group. Nevertheless, for a Legendrian knot~$K$ in the standard tight~$\!\left(S^3, \xi_{std}\right)$, we characterise the Legendrian motion group in terms of the \emph{contact mapping class group}~$\cmcgext{K}$: the group of contact isotopy classes of contactomorphisms of the Legendrian knot exterior that fix the boundary. Recall that a Legendrian knot is \emph{Legendrian invertible} if it is contact-isotopic to itself with the opposite orientation, and \emph{Legendrian non-invertible} if there is no such contact-isotopy.

\newpage 

\begin{customthm}{C}[\ref{thm:legendrian motions as contact mapping classes}]\label{customthm:legendrian motions as contactomorphisms of exterior}
    Let~$K$ be a Legendrian knot in~$\std$.
    \begin{enumerate}[(a)]
        \item If~$K$ is Legendrian non-invertible, then there is an isomorphism
        \[\Lambda \colon \lmotionstd{K} \xlongrightarrow{\cong} \cmcgext{K}.\]
        \item If instead~$K$ is Legendrian invertible, then there is an isomorphism
        \[\Lambda \colon \lmotionstd{K} \xlongrightarrow{\cong} \cmcgext{K} \rtimes \faktor{\mathbb{Z}}{2\mathbb{Z}}.\]
    \end{enumerate}
\end{customthm}

In other words, non-triviality of Legendrian motions is detected by non-triviality of contact mapping classes of the knot exterior, and vice versa. We view this as a moral substitute for the Dahm homomorphism in the Legendrian setting. As an application, we discuss how one may view the `K\'alm\'an loop' \cite{Kalman} (a loop of Legendrian torus knots which is smoothly null-homologous and yet contact isotopically non-contractible) as a contact mapping class, as first pointed out by Min--Fern\'andez \cite{Min-Fernandez}.

The proof of \emph{Theorem \ref{customthm:legendrian motions as contactomorphisms of exterior}} mixes various homotopy-theoretic results of Eliashberg--Mishachev, Fern\'andez--Min, and Fern\'andez--Mart\'inez-Aguinaga--Presas \cite{ Eliashberg-Mishachev, Min-Fernandez, Fernandez-Aguinaga-Presas} with the following proposition, which is one of our central technical steps.
\begin{customprop}{D}[\ref{prop:param_to_unparam}]\label{customprop:param to unparam}
    Let~$(M,\xi)$ be a contact 3-manifold and~$K$ be a Legendrian knot in~$(M,\xi)$. Then the fibration
    \[\diff\!\left(S^1\right) \hookrightarrow \lemb\!\left(K,(M, \xi)\right)  \rightarrow  \lsub\!\left(K,(M, \xi)\right)\]
    is locally trivial. 
\end{customprop}
\noindent One may ask: to what extent does a similar argument work in the smooth case? The smooth analogue of \emph{Proposition \ref{customprop:param to unparam}} was classically established by Cerf and Palais \cite{CerfTopologiePlongements61, PalaisLocalTriviality60} (see also Wattenberg \cite{WattenbergSmoothMotionsUnlink72}). Combining it with work of Budney and Cohen \cite{Budney-Cohen, Budney} gives tools with which one can hope to implement the same strategy, thus reinterpreting smooth motions as mapping classes of the knot exterior. However, it seems that a result as clean as \emph{Theorem \ref{customthm:legendrian motions as contactomorphisms of exterior}} is only possible thanks to the rigidity of the contact-geometric setting. We briefly discuss this smooth implementation, and possible ramifications for the Dahm homomorphism for smooth knots in~$S^3$, in the Appendix.

\subsection*{Outline}\label{subsec:outline}

In \S\ref{sec:background}, we give the necessary background on fibre sequences, Legendrian links and (smooth) motion groups. We introduce the Legendrian motion group in \S\ref{sec:Legendrian motion group}, and prove \emph{Theorem \ref{customthm:legendrian motion group as pi_1 of embedding space}} and \emph{Corollary \ref{customcor:legendrian isotopy implies isomorphism}}. In \S\ref{sec:monodromy maps} we prove \emph{Theorem \ref{customthm:legendrian motions as contactomorphisms of exterior}}, setting up the correspondence between Legendrian motions and the contact mapping classes. Appendix \ref{appendix:smooth_case} is concerned with translating the techniques of \S\ref{sec:monodromy maps} to the smooth case.

\subsection*{Acknowledgements}
We would like to thank Susanna Terron, Livio Ferretti, and Rachael Boyd for useful conversations, and Eduardo Fern\'andez for very helpful correspondence, including pointing out a key result for the proof of \emph{Lemma \ref{lem:local_contractibility}}. We thank Boyd, Fern\'andez, and Brendan Owens for their comments on a draft of this paper. We thank our doctoral supervisors Andy Wand \& Ana Lecuona (IN) and Rachael Boyd, Tara Brendle, \& Brendan Owens (GC) for guidance and support. 

IN was supported by EPSRC Studentship No. ~EP/W524359/1, and GC by a College Scholarship from the College of Science and Engineering, University of Glasgow. 

\section{Fibre sequences, Legendrian links, and smooth motion groups}\label{sec:background}

 In this section we set up some necessary homotopy theory technology, and give the necessary preliminaries on Legendrian links and the classical smooth motion group. We highlight that all the function spaces appearing in this article are equipped with the Whitney~$C^\infty$-topology, and have the homotopy type of~$CW$-complexes. 

\subsection{$G$-locally retractile spaces and the Cerf--Palais criterion}\label{subsec:G-locally retractile}

We set out some classical results of Palais \cite{PalaisLocalTriviality60} (see also \cite{LimaLocalTrivialityRestriction64}), proven independently and in greater generality by Cerf \cite{CerfTopologiePlongements61}, which will enable us to conclude that certain maps are locally trivial fibrations. This exposition follows that of Cantero Mor\'an--Randal-Williams \cite[\S2]{Cantero-Moran_Randal-Williams} and Boyd--Bregman--Steinebrunner \cite[\S2.2]{Boyd-Bregman-Steinebrunner}. 

\begin{definition}\label{def:G-locally retractile}
    A space~$X$ with an action of a topological group~$G$ is \emph{$G$-locally retractile} if, for each~$x \in X$, there is a neighbourhood~$U \subseteq X$ of~$x$ and a map~$\rho \colon U \to G$ such that for all~$u \in U$, we have~$\rho(u) \cdot x = u$. 
\end{definition}

We will rely on the \emph{Cerf--Palais criterion}, which is the first part of the following lemma.

\begin{lemma}[\cite{Cantero-Moran_Randal-Williams}, 2.5 \& 2.6]\label{lem:properties of G-locally retractile spaces}
    Let~$X$ be a~$G$-locally retractile space.
    \begin{enumerate}[(i)]
        \item Any~$G$-equivariant map~$f \colon Y \to X$ is a locally trivial fibration. \label{lemitem:Cerf-Palais criterion}
        \item If a~$G$-equivariant map~$g \colon X \to Z$ has local sections, then~$Z$ is also~$G$-locally retractile. \label{lemitem:equivariant map with local sections implies target is locally retractile}
    \end{enumerate}
\end{lemma}

Combining these statements, we obtain: if~$X$ is a~$G$-locally retractile space, then a~$G$-equivariant map~$X \to Z$ with local sections is in fact a locally trivial fibration.

\subsection{Legendrian links in~$(M,\xi)$}
Here we discuss only the essential background for the contact geometric notions we require; we refer the reader to \cite{Etnyre_lecture_notes, Geiges}, for a more extensive treatment. We recall the Legendrian embeddings and some of their standard properties. Throughout the paper we will make clear whether results stated for knots extend to links or not.

A \emph{contact structure} on an oriented 3-manifold~$M$ is a maximally non-integrable plane field~$\xi$ in the tangent bundle of~$M$. A contact structure~$\xi$ is \emph{coorientable} if the normal bundle is orientable; that is, there exists an orientable line field transverse to~$\xi$. If~$\xi$ is cooriented, one defines~$\xi$ as the kernel of a 1-form~$\alpha$ which satisfies the \emph{contact condition}~$\alpha \wedge d\alpha >0$ (this is for \emph{positive} contact structures; the contact condition for \emph{negative} ones is~$\alpha \wedge d\alpha <0$). 

For the remainder of this section, fix a cooriented contact structure~$\xi$ on a closed, connected 3-manifold~$M$.

\begin{definition}\label{def:Legendrian-knot}
    Let~$(M, \xi)$ be a contact 3-manifold. An embedding~$\gamma \colon S^1 \hookrightarrow M$ is said to be \emph{Legendrian} if it is tangent to the contact structure at every point; that is, if~$d\gamma(\theta) \in \xi_{\gamma(\theta)}$ for all~$\theta \in S^1$. A \emph{Legendrian knot} is the image of such an embedding. 
\end{definition} 

A very important result in contact topology is the following proposition, which can be obtained from the work of Fuchs--Tabachnikov \cite{Fuchs_Tabachnikov}. See also \cite[Theorem 2.5]{EtnyreBook}
\begin{proposition}\label{prop:legendrian_approx}
    Any smooth isotopy class of knots in a contact manifold~$(M,\xi)$ admits a Legendrian representative.
\end{proposition}

\emph{Proposition \ref{prop:legendrian_approx}} has important consequences. When considering a smooth knot type, i.e. a connected component in the space of smooth knot embeddings in~$M$, we know that it contains \emph{at least} some Legendrian knot type, i.e. a connected component of Legendrian embeddings in~$(M,\xi)$.  Distinct values of the \emph{Thurston--Bennequin number} ($\tb$) and \emph{rotation number} ($\rot$) correspond to distinct Legendrian representatives of a smooth knot type (see \cite{Etnyre_lecture_notes} for definitions). In general not all pairs~$(\tb, \rot)$ can be realised as invariants of a Legendrian knot due to the \emph{Thurston-Bennequin inequality} \cite{Eliashberg_inequality}. There are however plenty of examples where multiple pairs are realised (e.g. for the unknot in~$(\mathbb{R}^3,\xi_{std})$). Thus the natural forgetful map 
\[\lemb\!\left(K,(M,\xi)\right) \to \emb(K,M)\]
is surjective but generically (highly) non-injective. 

Contact structures come in two flavours: \emph{tight} and \emph{overtwisted}. The behaviours of these two classes are quite distinct, but in this paper we tend not to make any distinction, as we are interested in Legendrian embeddings in general and less in the specific geometry of the contact structure. Typically however, when working with~$S^3$ and~$\mathbb{R}^3$, we will be using their standard tight structures, defined as follows.

The standard contact structure on~$\mathbb{R}^3$ is given by~$$\xi_{std} \coloneqq \ker(dz - ydx),$$ whereas the standard contact
structure on~$S^3 \subset \mathbb{C}^2$ is given by~$$\xi_{std} \coloneqq TS^3 \cap JTS^3,$$
where~$J$ denotes the standard complex structure on~$\mathbb{C}^2$. These two structures are closely related:~$(\mathbb{R}^3,\xi_{std})$ is contactomorphic to~$(S^3\setminus \{p\},\xi_{std}|)$,
for any choice of point~$p \in S^3$ \cite[Proposition 2.1.8]{Geiges}.

\begin{definition} \label{definition:contact_vector_field}
A \emph{contact vector field}~$v$ is a vector field whose flow preserves~$\xi$. Precisely, if~$\xi=\ker(\alpha)$, then the flow~$\phi_t$ of~$v$ preserves~$\xi$ if we have~$\phi_t^*\alpha=f\alpha$ for some non-vanishing function~$f$.
\end{definition}
\begin{definition}\label{def:contactomorphism} 
    Let~$(M,\xi)$ be a contact 3-manifold and~$f$ a self-diffeomorpism of~$M$. We say that~$f$ is a \emph{contactomorphism} (a.k.a. contact-diffeomorphism) if~$f$ preserves~$\xi$, i.e.~$f^*\xi=\xi$. The subgroup~$\cont(M, \xi) \leq \diff(M)$ consisting of all such is the \emph{contactomorphism group} of~$(M, \xi)$. It is endowed with the subspace topology.
\end{definition}

\subsection{The smooth motion group}\label{sec:intro of smooth motion group}

In this subsection, we define motion groups and recall some relevant theorems. For some additional background, see \cite{GoldsmithMotionGroups81, WattenbergSmoothMotionsUnlink72}. In what follows,~$M$ is a smooth manifold,~$N\subset M$ a smooth compact submanifold,~$\diff_c(M)$ denotes the group of compactly supported diffeomorphisms of~$M$, and~$\diff_c(M, N)$ is the subgroup of those diffeomorphisms preserving~$N$ setwise.

\begin{definition}\label{def:smooth motion}
    A \emph{motion of~$N$ in~$M$} is a path~$f = \{f_t\}_{t \in [0,1]}$ in~$\diff_c(M)$ such that~$f_0 = \id_M$ and~$f_1 \in \diff_c(M, N)$   
\end{definition}

\begin{definition}\label{def:stationary_motion}
A motion~$f$ is \emph{stationary} if~$f_t \in \diff_c(M, N)$ for all~$t \in [0, 1]$.
\end{definition}

The \emph{product}~$g \circ f$ of two motions~$f, g$ of~$N$ in~$M$ is given by
\begin{equation*}
    (g \circ f)_t =
    \begin{cases}
        f_{2t} & 0 \leq t \leq \frac{1}{2} \\
        g_{2(t-\frac{1}{2})} \circ f_1 & \frac{1}{2} \leq t \leq 1.
    \end{cases}
\end{equation*}

In other words: do~$f$ twice as fast, and then do~$g$ twice as fast after translating by~$f_1$. The \emph{inverse}~$f^{-1}$ of a motion~$f$ is given by~$(f^{-1})_t = f_{1-t} \circ f_1^{-1}$.

\begin{definition}\label{def:equivalent_motions}
Two motions~$f, g$ are \emph{equivalent} (written~$f \equiv g$) if~$g^{-1} \circ f$ is homotopic rel endpoints to a stationary motion. 
\end{definition}

The set of equivalence classes of motions forms a group under the product defined above. This is simply the relative fundamental group~$\pi_1 \!\left( \diff_c(M), \diff_c(M, N), \id_M \right)$. Although \emph{a priori} the relative fundamental group is simply a \emph{set} of homotopy classes of paths, this set can be given a group structure in two ways, since~$\diff_c$ is a topological group. One is the product just defined; the other is by pointwise composition~$(g \cdot f)_t = g_t \circ f_t$. One can show that the two paths~$g \cdot f$ and~$g \circ f$ are homotopic rel endpoints. 

\begin{definition}\label{def:smooth motion group}
    The \emph{(smooth) motion group of~$N$ in~$M$} is the relative fundamental group 
    \[\motion{N}{M} \coloneqq \pi_1 \!\left( \diff_c(M), \diff_c(M, N), \id_M \right).\]
\end{definition}

The equivalence relation on motions can be rephrased as follows \cite[Proposition 2.8]{GoldsmithMotionGroups81}. Let~$f, g$ be motions of~$N$ in~$M$. Then~$f \equiv g$ if and only if~$f$ is homotopic to a motion~$f'$ of~$N$ such that~$f'_t(N) = g_t(N)$ for all~$t \in [0, 1]$.

Smooth motion groups are helpfully reinterpreted as fundamental groups of embedding spaces. Let~$\emb(N, M)$ denote the space of smooth embeddings of~$N$ into~$M$. Equipping~$\emb(N, M)$ with the~$C^\infty$-Whitney topology makes it a principal~$\diff(N)$-bundle, where~$\diff(N)$ acts by precomposition. The quotient 
\[\sub(N, M) \coloneqq \faktor{\emb(N, M)}{\diff(N)}\]
is the \emph{unparametrised embedding space}. 

Wattenberg \cite{WattenbergSmoothMotionsUnlink72} showed that the motion group is in fact isomorphic to the fundamental group of the unparametrised embedding space. Precisely: fix a favoured embedding~$i \colon N \hookrightarrow M$ as a basepoint of~$\emb(N, M)$, and its image~$i(N)$ as a basepoint of~$\sub(N, M)$. Let~$\calE(i)$ denote the connected component of~$\sub(N, M)$ which contains the basepoint~$i(N)$. 

\begin{theorem}[\cite{WattenbergSmoothMotionsUnlink72}, Lemma 1.4]\label{thm:smooth motion group is pi_1 of embedding space Wattenberg}
   ~$\motion{N}{M} \cong \pi_1\!\left(\calE(i), i(N)\right)$.
\end{theorem}

\section{The Legendrian motion group}\label{sec:Legendrian motion group}

The goal of this section is to define the \emph{Legendrian motion group} of a Legendrian knot~$L$ in a contact 3-manifold~$(M,\xi)$.
We will then prove \emph{Theorem \ref{customthm:legendrian motion group as pi_1 of embedding space}} (\ref{thm:legendrian motion group as pi_1 of embedding space}), that it can be viewed as the fundamental group of the space of unparametrised Legendrian embeddings of the knot into~$M$. For ease of notation, the definitions and theorems in this section are stated in terms of knots, but they can be easily extended to links as well.

\subsection{Definition of the Legendrian motion group}\label{subsec:def of Leg motion group}

We attempt to mimic the smooth construction. Let~$(M, \xi)$ be a closed, connected contact 3-manifold, and let~$K \subseteq (M, \xi)$ be a Legendrian knot. Let~$\cont(M, \xi)$ denote the space of contactomorphisms (see \emph{Definition \ref{def:contactomorphism}}) of~$(M, \xi)$, and~$\cont\!\!\left((M, \xi), K\right)$ the subspace consisting of those contactomorphisms which fix~$L$ setwise. Let~$\cont_c(M, \xi)$ and~$\cont_c\!\!\left((M, \xi), K\right)$ denote the respective subspaces of contactomorphisms with compact support. 

\begin{definition}\label{def:Legendrian_motion}
    A \emph{Legendrian motion} of~$K$ in~$(M,\xi)$ is a path~$f= \{f_t\}_{t\in[0,1]}$ in~$\cont_c(M,\xi)$ such that~$f_0=id_M$ and~$f_1 \in \cont_c((M,\xi),K)$.
\end{definition}

That is, a Legendrian motion of~$K$ is a path of contactomorphisms of~$(M, \xi)$ which begins at the identity and finishes at a contactomorphism which preserves~$K$ setwise.

\begin{definition}\label{def:leg_stationary_motion}
    A Legendrian motion~$f$ is \emph{stationary} if~$f_t \in \cont_c((M,\xi),K)$ for all~$t \in [0,1]$.
\end{definition}

The product~$g \circ f$ of two Legendrian motions is defined as in the smooth case: \begin{equation*}
    (g \circ f)_t =
    \begin{cases}
        f_{2t} & 0 \leq t \leq \frac{1}{2} \\
        g_{2(t-\frac{1}{2})} \circ f_1 & \frac{1}{2} \leq t \leq 1.
    \end{cases}
\end{equation*}
Likewise, the \emph{inverse}~$f^{-1}$ of a Legendrian motion~$f$ is given by~$\!\left(f^{-1}\right)_t = f_{1-t} \circ f_1^{-1}$.

\begin{definition}\label{def:leg_equivalent_motions}
    Two Legendrian motions~$f, g$ are said to be \emph{equivalent} (written~$f \equiv g$) if~$g^{-1} \circ f$ is homotopic rel endpoints to a stationary Legendrian motion.
\end{definition}
The set of equivalence classes of Legendrian motions forms a group with the product defined above: it is simply the relative fundamental group~$\pi_1 \!\left( \cont_c(M,\xi), \cont_c((M,\xi),K),\id_K \right)$. As in the smooth case the group structure on the topological group~$\cont_c(M, \xi)$ is, up to homotopy rel endpoints, the same group structure as the one given by pointwise composition.

\begin{definition}\label{def:legendrian motion group}
    The \emph{Legendrian motion group} (of~$K$ in~$(M,\xi)$) is the relative fundamental group 
    \[\lmotion{K}{(M, \xi)} \coloneqq \pi_1 \!\left( \cont_c(M,\xi), \cont_c\!\!\left((M,\xi), K\right),\id_M \right).\]
\end{definition}

\begin{proposition}\label{proposition:homotopic_iff_equivalent}
    Let~$f,g$ be Legendrian motions of~$K$ in~$M$. Then~$f \equiv g$ if and only if~$f$ is homotopic to~$f'$, where~$f'$ is a Legendrian motion of~$K$ in~$M$ such that for all~$t$,~$f'_t(K)=g_t(K)$.  
\end{proposition}

\begin{proof}
    We have~$f \equiv g$, so~$f$ is homotopic to~$g \circ s$, where~$s$ is a stationary Legendrian motion. Let~$f'= g \circ s$. Then,~$f'_t(K)=g_t(K)$ for all~$t$. On the other hand, if~$f$ is homotopic to~$f'$, with~$f'_t(K)=g_t(K)$ for all~$t$, then~$s=g^{-1}\circ f$ is a stationary Legendrian motion, so~$f \equiv g$.
\end{proof}

\subsection{Legendrian motion group as a fundamental group}\label{subsec:Leg motion group is pi_1 of Leg emb space}

In this subsection we prove \emph{Theorem \ref{customthm:legendrian motion group as pi_1 of embedding space}} (\ref{thm:legendrian motion group as pi_1 of embedding space}). The main technical engine here is \emph{Theorem \ref{thm:parametrised_fiber_bundle}}, which we then combine with classical arguments adapted to the Legendrian case. 

Let~$\lemb\!\!\left(K,(M, \xi)\right)$ denote the space of Legendrian embeddings~$K \hookrightarrow (M,\xi)$. Equipped with the~$C^\infty$ Whitney topology, this is a principal~$\diff\!\left( S^1\right)$-bundle where~$\diff\!\left( S^1\right)$ acts by precomposition; we call the quotient
\[\lsub\!\left(K,(M, \xi)\right) \coloneqq \faktor{\lemb\!\left(K,(M, \xi)\right)}{\diff\!\left(S^1\right)}\]
the \emph{unparametrised Legendrian embedding space}. 

\begin{remark}\label{remark:orientation}
    For us, the word ``unparametrised'' means the Legendrian submanifold is also unoriented. This differs from the usual convention in the contact setting \cite{Aguinaga,Aguinaga-Presas-Fernandez}, where the space of Legendrian embeddings is often instead quotiented by~$\diff^+(S^1)$ (so endowing Legendrian submanifolds with a fixed orientation). 
\end{remark}

Fix a favoured Legendrian embedding~$\iota \colon K \to M$ to be the basepoint of~$\lemb\!\left(K,(M, \xi)\right)$ and its image~$\iota(K)$ to be our chosen basepoint of~$\lsub\!\left(K,(M, \xi)\right)$, and let~$\mathcal{LE}(\iota)$ be the connected component of~$\lsub\!\left(K,(M, \xi)\right)$ which contains~$\iota(K)$.

\begin{theorem}[\ref{customthm:legendrian motion group as pi_1 of embedding space}]\label{thm:legendrian motion group as pi_1 of embedding space}
    The Legendrian motion group~$\lmotion{K}{(M, \xi)}$ is isomorphic to~$\pi_1\!\left(\mathcal{LE}(\iota)\right)$.
\end{theorem}

The proof of this theorem has a number of parts, so we relegate the final piecing-together until the end of \S\ref{subsec:Leg motion group is pi_1 of Leg emb space}. Along the way, we prove some contact analogues of useful classical smooth results, which may be of independent interest.

\subsubsection{Step 1: parametrised embeddings}\label{subsubsec:step 1: parametrised embeddings}

We begin with the following lemma, which is a translation of \cite[Corollary 6.2.4]{Banyaga} to the setting of Legendrian embeddings.

\begin{lemma}\label{lem:local_contractibility}
    The space~$\lemb\!\left(K,(M,\xi)\right)$ is locally contractible.
\end{lemma}

\begin{proof}
    Fix the inclusion~$\iota$ as our preferred embedding. As shown in \cite[Example 2.5.11]{Geiges}, the Legendrian submanifold~$K$ has a neighbourhood in~$M$ that is contactomorphic to a neighbourhood of the zero section of~$\left(J^1(K), \xi_{jet}\right)$, the space of~$1$-jets on~$K$. The zero section corresponds to the Legendrian submanifold~$K$. 
    Any Legendrian embedding sufficiently close to~$\iota$ will have image lying in this model neighbourhood. Consequently, it is enough to study the case where~$(M,\xi)$ is~$(J^1(K),\xi_{jet})$ and~$\iota$ is the zero section. We must prove that this there is a contractible neighbourhood around~$\iota$.

    Fix a Legendrian embedding~$e$ which is~$C^1$-close to~$\iota$ in~$\lemb\!\left(K,(J^1(K),\xi_{jet})\right)$. In particular,~$e(K)$ is~$C^1$-close to~$K$ as Legendrian submanifolds of~$J^1(K)$.~$C^1$-closeness guarantees that~$e(K)$ is the graph of some section of the 1-jet bundle \cite[Proof of Theorem 6.2.3]{Banyaga}.
    We can now appeal to a result due to Lycagin \cite[1.2.4]{Lycagin} (see also \cite[Theorem 6.2.1]{Banyaga}
    ), which states that a section of the 1-jet space~$J^1(K)$ is Legendrian for~$\xi_{jet}$ if and only if it is the 1-jet of a function~$f\colon K \to \mathbb{R}$, with~$f$ being~$C^2$-close to the zero function.
    In particular, accounting for a small reparametrisation of~$e$ to match this section, a Legendrian embedding that is sufficiently close to~$\iota$ corresponds to a map that is close to~$(\id,0)$ in~$\diff(K) \times C^{\infty} (K,\mathbb{R})$, where~$0$ is the constantly vanishing function corresponding to the zero section. 
    We can reverse the process, starting with some~$(\phi, f') \in \diff(K) \times C^{\infty}(K, \mathbb{R})$ where~$\phi$ is~$C^1$-close to~$\id$ and~$f'$ is~$C^2$-close to~$0$, and obtaining a Legendrian embedding~$C^1$-close to~$\iota$. These two processes are inverses to each other.

    Hence we have a homeomorphism from a neighbourhood~$U_{\iota}$ of~$\iota$ to a neighbourhood~$V_{(\id,0)}$ of~$(\id,0)$. Now,~$\diff(K)$ is locally contractible, and the space of~$C^{\infty}$ functions from~$K$ to~$\mathbb{R}$ is convex. Hence, there is a (potentially smaller) contractible neighbourhood of~$(\id,0)$ contained in~$V_{(\id,0)}$. This gives the desired result once we pull-back via the homeomorphism.
\end{proof}

\begin{theorem}\label{thm:parametrised_fiber_bundle}
    Let~$(M, \xi)$ be a contact 3-manifold, and let~$K$ be an embedded Legendrian knot. Fix a Legendrian embedding~$\gamma \in \lemb\!\left(K,(M, \xi)\right)$ of~$K$ into~$M$. Then the map
    \begin{eqnarray*}
        p \colon \cont_c(M, \xi) & \to & \lemb\!\left(K,(M, \xi)\right) \\
        \phi & \mapsto & \phi \circ \gamma
    \end{eqnarray*}
    is a locally trivial fibration with fibre~$\cont_c^{\gamma(K)}(M, \xi)$, the subgroup of contactomorphisms of~$(M, \xi)$ (with compact support) which fix~$\gamma(K)$ \emph{pointwise}. 

    In particular, the parametrised Legendrian embedding space~$\lemb\!\left(K,(M, \xi)\right)$ is~$\cont_c(M, \xi)$-locally retractile.
\end{theorem}

\begin{proof}
    We prove that~$\lemb\!\left(K, (M, \xi)\right)$ is~$\cont_c(M, \xi)$-locally retractile, which allows us to invoke the Cerf--Palais criterion, \emph{Lemma \ref{lem:properties of G-locally retractile spaces}(\ref{lemitem:Cerf-Palais criterion})}. 
    
    Moving in~$\lemb\!\left(K, (M, \xi)\right)$ by precomposition with contactomorphisms, start with~$K$ viewed as a submanifold of~$M$ and take~$\gamma$ to be the inclusion embedding~$\iota \in \lemb\!\left(K,(M, \xi)\right)$. 

    Pick a contractible neighbourhood~$U$ of~$\iota$ in~$\lemb\!\left(K, (M, \xi)\right)$, which we can do by \emph{Lemma \ref{lem:local_contractibility}}.
    For any embedding~$e \in U$, take a path~$\{e_t\}_{t \in [0, 1]}$ in~$U$ connecting~$e_0 \coloneqq \iota$ to~$e_1 \coloneqq e$. Define a time-dependent vector field~$X^e = \{X_t^e\}_{t \in [0, 1]}$ by
    \begin{eqnarray*}
        X_t^e \colon e_t(K) & \to & TM \\
        e_t(\theta) & \mapsto & \odv{}{s}_{s = t}\!\left(e_s(\theta)\right).
    \end{eqnarray*}

    Define a family of Hamiltonian functions~$\{H^e_t\}_{t \in [0, 1]}$ by~$H^e_t = \alpha(X^e_t)$, where~$\alpha$ is the 1-form associated to~$(M, \xi)$. Precisely, 
    \begin{eqnarray*}
        H^e_t \colon e_t(K) & \to & \mathbb{R} \\
        e_t(\theta) & \mapsto & \alpha\!\left(X^e_t\!\left(e_t(\theta)\right)\right).
    \end{eqnarray*}

    Following the proof of (\cite{Geiges}, Theorem 2.6.2), we can extend the family~$\{H^e_t\}_{t \in [0, 1]}$ to a family~$\{\widetilde{H}^e_t\}_{t \in [0, 1]}$ of compactly supported smooth functions, such that each~$\widetilde{H}^e_t \colon M \to \mathbb{R}$ has Hamiltonian vector field~$\widetilde{X}^e_t$ (that is,~$\widetilde{H}^e_t = \alpha(\widetilde{X}^e_t)$). For each~$t \in [0, 1]$, the restriction of~$\widetilde{X}^e_t$ and~$\widetilde{H}^e_t$ to~$e_t(K)$ will agree with~$X^e_t$ and~$H^e_t$ respectively. We now have a time-dependent contact vector field~$\widetilde{X}^e = \{\widetilde{X}^e_t\}_{t \in [0, 1]}$ defined on all of~$M$. Since~$U$ is contractible, we have a contractible choice of paths~$\{e_t\}$, so we can assume this assignment varies smoothly in~$e \in U$. Integrating, we obtain a family of contactomorphisms~$\{\widetilde{\varphi}^e_t\}_{t \in [0, 1]}$ of~$(M, \xi)$ such that for each~$t \in [0, 1]$ we have
    \[\widetilde{\varphi}^e_t \circ e_0 = e_t.\]
    
    We aim to use this (time-dependent) flow to define a local section~$U \to \cont(M, \xi)$ which takes~$e = e_1$ to~$\varphi^e_1$. We would like to use an exponential map (mimicking the work of \cite{PalaisLocalTriviality60}), but such is ill-defined for time-dependent vector fields. We therefore take a detour via the canonical time-independent vector field~$\overline{X}^e$ associated to~$\widetilde{X}^e$, for which we can define an exponential map. 

    Concretely, define the time-independent vector field
    \begin{align*}\tag{$\ast$}\label{eq:time_dep_to_time_indept}
        \overline{X}^e \colon [0, 1] \times M & \to T\!\left([0, 1] \times M\right)\\
        (t, x) & \mapsto \!\left(1, \widetilde{X}^e_t(x)\right).
    \end{align*}
    We will view the time-dependent flow~$\{\widetilde{\varphi}^e_t\}_{t \in [0, 1]}$ as a flow of the time-independent vector field~$\overline{X}^e$ by the following trick. Define
    \begin{eqnarray*}
        \overline{\varphi}^e \colon D \subseteq \mathbb{R} \times \!\left([0, 1] \times M\right) & \to & [0, 1] \times M \\
        \!\left(s, (t, x)\right) & \mapsto & \overline{\varphi}^e_s(t, x) \coloneqq \!\left(s+t, \widetilde{\varphi}^e_{s+t}\!\left(\left(\widetilde{\varphi}^e_t\right)^{-1}(x)\right)\right),
    \end{eqnarray*}
    where~$D$ is the flow domain on which this map makes sense. Now,~$\overline{\varphi}^e$ is the time-independent flow of~$\overline{X}^e$ if and only if~$\widetilde{\varphi}^e$ is the time-dependent flow of~$\widetilde{X}^e$ (that is,~$\overline{\varphi}^e$ is the unique solution of the time-independent initial value problem if and only if~$\widetilde{\varphi}^e$ is the unique solution of the time-dependent initial value problem). 

    By definition, for any~$(t, x) \in [0, 1] \times M$, we have 
    \[\Exp_{(t, x)}\!\left(s \cdot \overline{X}^e(t, x)\right) = \overline{\varphi}^e_s\!\left(t, x\right) = \!\left(s+t, \widetilde{\varphi}^e_{s+t}\!\left(\left(\widetilde{\varphi}^e_t\right)^{-1}(x)\right)\right).\]
    In particular, taking~$t = 0$ and~$x = e_0(\theta)$ for~$\theta \in S^1$, we have
    \begin{eqnarray*}
        \Exp_{\left(0, e_0(\theta)\right)}\!\left(s \cdot \overline{X}^e(0, e_0(\theta))\right) & = & \!\left(s, \widetilde{\varphi}^e_{s}\!\left(\left(\widetilde{\varphi}^e_0\right)^{-1}\!\left(e_0(\theta)\right)\right)\right) \\
        \implies \Exp_{\left(0, e_0(\theta)\right)}\!\left(s \cdot \!\left(1, \widetilde{X}^e_0\!\left(e_0(\theta)\right)\right)\right) & = & \!\left(s, \widetilde{\varphi}^e_{s}\!\left(e_0(\theta)\right)\right) \\
        & = & \!\left(s, e_s(\theta)\right)
    \end{eqnarray*}
    and in particular
    \[\Exp_{\left(0, e_0(\theta)\right)}\!\left(\overline{X}^e_0\!\left(e_0(\theta)\right)\right) = \!\left(1, e_1(\theta)\right).\]

    Let~$N$ be a compact neighbourhood of~$e_0(K)=\iota(K)$ in~$M$ on which the time-dependent flow~$\{\widetilde{\varphi}^e_t\}_{t \in [0, 1]}$ is supported. Write~$X_N(TM)$ for the space of vector fields on~$M$ supported on~$N$ and write~$X_N\!\left(T\!\left([0, 1] \times M\right)\right)$ for the space of vector fields on~$[0, 1] \times M$ which are supported on~$[0, 1] \times N$. 
    
    We define the map 
    \begin{eqnarray*}
        \overline{S} \colon U & \to & X_N(T([0,1]\times M)) \\
        e & \mapsto & \overline{X}^e
    \end{eqnarray*}
    as the composition of the above Hamiltonian map and the assignment in (\ref{eq:time_dep_to_time_indept}). Since the choice of~$\widetilde{X}^e$ varied smoothly in~$e$, we have that~$\overline{S}$ is smooth.
    Define the map
    \[\overline{E} \colon X_N\!\left(T\!\left([0, 1] \times M\right)\right) \to \diff([0, 1] \times M)\]
    by~$\overline{E}(\overline{Y})(s, x) = \Exp_{(0, x)}(s \cdot \overline{Y}(0, x))$ where this assignment is a diffeomorphism, and smoothed off to the identity elsewhere. 

    Let~$\operatorname{LPDiff}([0, 1] \times M)$ be the subset of~$\diff([0, 1] \times M)$ consisting of level-preserving diffeomorphisms; that is, diffeomorphisms~$g$ such that for each~$s \in [0, 1]$, we have~$g(\{s\} \times M) \subseteq \{s\} \times M$. Define the map
    \[\pi_2 \colon \operatorname{LPDiff}([0, 1] \times M) \to \diff(M)\]
    by, for any~$g \in \operatorname{LPDiff}([0, 1] \times M)$,
    \begin{eqnarray*}
        \pi_2(f) \colon M & \to & M \\
        x & \mapsto & p_2(g(1, x)),
    \end{eqnarray*}
    where~$p_2 \colon [0, 1] \times M \to M$ is the projection. 
    
    We are now ready to define our local section as
    \[\sigma \coloneqq \pi_2 \circ \overline{E} \circ \overline{S} \colon U \to \cont_c(M, \xi).\]
    We first show that this is well-defined. Let~$e \in U$; we need~$\left(\overline{E} \circ \overline{S}\right)(e)$ to be a level-preserving diffeomorphism of~$[0, 1] \times M$ such that the time-1 diffeomorphism is actually a contactomorphism. Let~$(s, x) \in [0, 1] \times M$; now
    \begin{eqnarray*}
        \left(\left(\overline{E} \circ \overline{S}\right)(e)\right)(s, x) & = & \overline{E}\left(\overline{X}^e\right)(s, x) \\
        & = & \Exp_{(0, x)}\!\left(s \cdot \overline{X}^e(0, x)\right) \\
        & = & \overline{\varphi}^e_s(0, x) \\
        & = & \left(s, \widetilde{\varphi}^e_s\!\left((\widetilde{\varphi}^e_0)^{-1}(x)\right)\right) \in \{s\} \times M.
    \end{eqnarray*}
    This verifies both requirements, since the~$\widetilde{\varphi}^e_s$ are contactomorphisms. Hence~$\sigma$ is a well-defined map. Furthermore from this calculation, we have, for any~$\theta \in S^1$,
    \begin{eqnarray*}
        (p \circ \sigma)(e)(\theta) & = & \left(\sigma(e) \circ \iota\right)(\theta) \\
        & = & (\sigma(e))(e_0(\theta)) \\
        & = & \widetilde{\varphi}^e_1\!\left((\widetilde{\varphi}^e_0)^{-1}(e_0(\theta))\right) \\
        & = & \widetilde{\varphi}^e_1(e_0(\theta)) = e(\theta),
    \end{eqnarray*}
    so~$p \circ \sigma = \id_U$. It is clear that~$\sigma$ is continuous, since~$\overline{S}$ was smooth. Hence~$\sigma$ is a local section of~$p$ near~$\iota \in \lemb\!\left(K, (M, \xi)\right)$. 
    
    We conclude that~$\lemb\!\left(K,(M, \xi)\right)$ is~$\cont_c(M, \xi)$-locally retractile. We can now apply the Cerf-Palais criterion to conclude that~$p$ is a locally trivial fibration, as required.
\end{proof}

\subsubsection{Step 2: Unparametrised embeddings}\label{subsubsec:step 2: unparametrised embeddings}

For our purposes, we will need to study the above fibration in the unparametrised setting; that is, we would like the statement of \emph{Theorem \ref{thm:parametrised_fiber_bundle}} to hold with~$\lemb\!\left(K,(M, \xi)\right)$ replaced by~$\lsub\!\left(K,(M, \xi)\right)$. In other words, we need to show that we can maintain the fibre bundle structure when we quotient out by the action of~$\diff(S^1)$.

We will adapt the classical smooth strategy (for example, as sketched in \cite[\S2.2]{Boyd-Bregman-Steinebrunner}) to the contact case. We need to prepare some working tools. As in the proof of \emph{Theorem \ref{thm:parametrised_fiber_bundle}}, we may move in the space of Legendrian embeddings by precomposition with some contactomorphism, so we start with a Legendrian submanifold~$K$ of~$M$ and take~$\gamma$ to be the inclusion~$\iota \in \lemb\!\left(K,(M, \xi)\right)$. Let~$(N_K,p_K,K)$ be a tubular neighborhood of~$K$ in~$(M,\xi)$. Here~$p_K$ denotes the submersion~$N_K \to K$.
  
We first specify a neighbourhood~$V_{\iota} \subseteq \lemb\!\left(K,(M,\xi)\right)$ of~$\iota$ as follows.

Start with 
\[A \coloneqq \{ j\in C^{\infty}(S^1,N_K)\ \colon p_K\circ j \in \emb(S^1,M) \}.\]
By \cite[Lemma 13.4]{MichorManifoldsDiffMappings80},~$A \subset \emb(S^1, M)$. Then define 
\[Q_{\iota} \coloneqq \{j \in A \colon p_K \circ j=\iota\}.\] As in \cite[13.5]{MichorManifoldsDiffMappings80}, we have~$Q_{\iota} \subset \emb(S^1, N_K)$. We then take the intersection 
\[V_{\iota} \coloneqq Q_{\iota} \cap \lemb\!\left(K,(M, \xi)\right).\]
This also defines a subset of~$\lemb\!\left(K,(N_K,\xi_\lvert)\right)$, where~$\xi_\lvert$ denotes the restriction of~$\xi$ to~$N_K$. 

Consider the quotient mapping~$u \colon \lemb\!\left(K,(M, \xi)\right) \to \lsub\!\left(K,(M, \xi)\right)$.
The following is a contact version of \cite[13.5-6]{MichorManifoldsDiffMappings80}.

\begin{lemma}\label{lem:prep for param_to_unparam prop in contact setting}
    \begin{enumerate}[(i)]
        \item The restriction
       ~$u\big|_{V_{\iota}} \colon V_{\iota} \to \lsub\!\left(K,(M, \xi)\right)$
        is injective.
        \item For any open subset~$W \subseteq \diff(S^1)$, the subset~$V_{\iota} \circ W$ is open in~$\lemb\!\left(K,(M, \xi)\right)$. \label{lempart:Michor 13.5.2 in contact setting}
        \item In the quotient topology,~$u(V_{\iota})$ is open in~$\lsub\!\left(K,(M, \xi)\right)$.
    \end{enumerate}
\end{lemma}

\begin{proof}
    \begin{enumerate}[(i)]
        \item If~$j, j' \in V_{\iota}$ with~$u(j)=u(j')$, then~$j=j' \circ g$ for some~$g \in \diff(S^1)$. Then 
        \[\iota_{K}=p_K\circ j= p_K\circ( j' \circ g)=(p_K\circ j') \circ g= \iota \circ g,\]
        so~$g=id_{S^1}$ and~$j=j'$.
        \item By \cite[13.5.2]{MichorManifoldsDiffMappings80},~$Q_{\iota} \circ W$ is open in~$\emb(S^1, M)$. Then
        \[V_{\iota}\circ W=(Q_{\iota} \cap \lemb\!\left(K,(M, \xi)\right)\circ W= (Q_{\iota}\circ W) \cap \lemb\!\left(K,(M, \xi)\right),\]
        and so~$V_{\iota} \circ W$ is open in~$\lemb\!\left(K,(M, \xi)\right)$ by the definition of subspace topology.
        \item Take~$W = \diff(S^1)$ in (\ref{lempart:Michor 13.5.2 in contact setting}). The inverse image of~$u(V_{\iota})$ under~$u$ is~$V_{\iota} \circ \diff(S^1)$, which is open in~$\lemb\!\left(K,(M, \xi)\right)$, so~$u(V_{\iota}) \subseteq \lsub\!\left(K,(M, \xi)\right)$ is open in the quotient topology.
    \end{enumerate}
\end{proof}

\begin{proposition}\label{prop:param_to_unparam}
    Let~$(M,\xi)$ be a contact 3-manifold and~$K$ be an (unoriented) Legendrian knot in~$(M,\xi)$. Then 
    \begin{eqnarray*}
        \lemb\!\left(K,(M, \xi)\right) & \longrightarrow & \lsub\!\left(K,(M, \xi)\right) \\
        \phi & \longmapsto & \phi(K)
    \end{eqnarray*}
    is a locally trivial fibration with fibre~$\diff(S^1)$. 
\end{proposition}

\begin{proof}
    Write~$\hat{\iota}_K$ and~$\widehat{V}_{\iota}$ for the respective images of~$\iota$ and~$V_{\iota}$ under the quotient map~$u$. By \emph{Lemma \ref{lem:prep for param_to_unparam prop in contact setting} (iii)},~$\widehat{V}_{\iota}$ is an open neighbourhood of~$\hat{\iota}_K$ in~$\lsub\!\left(K,(M, \xi)\right)$.

    We will show that the map
    \[\lemb\!\left(K,(M, \xi)\right)\to \lsub\!\left(K,(M, \xi)\right)\]
    is trivial over~$\widehat{V}_{\iota}$. We split the proof into six steps, following \cite[13.11]{MichorManifoldsDiffMappings80}.
    
    \begin{enumerate}
        \item Define~$s_{\iota}: \widehat{V}_{\iota} \to \lemb\!\left(K,(M, \xi)\right)$ by~$s_{\iota}=(u\big|_{{V_{\iota}}})^{-1}$. This is well defined, since~$u\big|_{V_{\iota}}$ is injective by \emph{Lemma \ref{lem:prep for param_to_unparam prop in contact setting} (i)}. Hence~$s_{\iota}$ is a local section of~$u$ around~$\iota \in \lsub\!\left(K, (M, \xi)\right)$.
        \item The fibres of~$u$ over~$\widehat{V}_{\iota}$ meet~$V_{\iota}$ in exactly one point by \emph{Lemma \ref{lem:prep for param_to_unparam prop in contact setting} (i)}. Since the action~$\rho$ of~$\diff(S^1)$ on~$\lemb\!\left(K,(M, \xi)\right)$ is free, the mapping~$\rho\big|_{\diff S^1 \times V_{\iota}} \colon \diff(S^1) \times V_{\iota} \to u^{-1}(\widehat{V}_{\iota})$ is bijective, so there is an inverse mapping 
        \[(\rho\big|_{\diff S^1 \times V_{\iota}})^{-1} \eqqcolon (\gamma_{\iota},\delta_{\iota}) \colon u^{-1}(\widehat{V}_{\iota}) \to \diff(S^1) \times V_{\iota}.\]
        Hence we have maps~$\gamma_{\iota} \colon u^{-1}(\widehat{V}_{\iota}) \to \diff(S^1)$ and~$\delta_{\iota} \colon u^{-1}(\widehat{V}_{\iota}) \to V_{\iota}$ such that for any element~$j \in u^{-1}(\widehat{V}_{\iota})$ we have
        \[\delta_{\iota}(j)\circ \gamma_{\iota}(j) = \rho(\gamma_{\iota}(j),\delta_{\iota}(j))=j\]
        and, since~$\delta_{\iota}(j) \in V_{\iota}$, we have~$p_K \circ \delta_{\iota}(j) = \iota$.
        \item \textbf{Claim.}~$\gamma_{\iota}$ is smooth. 
        \begin{proof}
            We have 
            \begin{eqnarray*}
                \iota \circ \gamma_{\iota}(j) & = & \!\left(p_K \circ \delta_{\iota}(j)\right) \circ \gamma_{\iota}(j) \\
                & = & p_K(j) \circ \!\left(\delta_{\iota}(j) \circ \gamma_{\iota}(j)\right) = p_K\circ j,
            \end{eqnarray*}
            by step (2). Hence~$\gamma_{\iota}(j) = \rho(\cdot,\iota)^{-1}\circ (p_K)_*(j)$, which means~$\gamma_{\iota} = \rho(\cdot,\iota)^{-1}\circ (p_K)_*$, and the latter is smooth. 
        \end{proof}
        \item \textbf{Claim.}~$\delta_{\iota}$ is smooth.
        \begin{proof}
            We have~$\delta_{\iota}(j) \circ \gamma_{\iota}(j)= j$, so~$\delta_{\iota}(j)=j \circ \gamma_{\iota}(j)^{-1}$ and ~$\delta_{\iota}=\rho \circ(inv \circ \gamma_{\iota}, \id)$, which is smooth (here we denote by~$inv$ the map that takes a diffeomorphism of~$S^1$ to its inverse).
        \end{proof}
        \item By the above steps,~$\rho$ is a diffeomorphism. This is our trivialization map.
        \item \textbf{Claim.} The map~$s_{\iota}: \widehat{V}_{\iota} \to V_{\iota}$ constructed in step (1) is continuous (and hence a homeomorphism onto its image).
        \begin{proof}
            For~$\hat{y} \in \widehat{V}_{\iota}$ we have~$$ \{ s_{\iota}(\hat{y})\} = \delta_{\iota}(u^{-1}(\hat{y}))$$ by construction. Let~$T \subset V_{\iota}$ be open. Then by step (4),~$\delta_{\iota}^{-1}(T)$ is open in~$u^{-1}(\widehat{V}_{\iota})$, which in turn is open in~$\lemb\!\left(K,(M,\xi)\right)$ , so 
            \[u^{-1}(s_{\iota}^{-1}(T))=u^{-1}(u(T))=\delta_{\iota}^{-1}(T)\]
            is open in~$\lemb\!\left(K,(M,\xi)\right)$. Now by the definition of quotient topology,~$s_{\iota}^{-1}(T)$ is open in~$\lemb\!\left(K,(M,\xi)\right)$. 
        \end{proof}
    \end{enumerate}

    We thus have constructed our local trivializations.
\end{proof}

\begin{theorem}\label{thm:unparametrised_fiber_bundle}
    Let~$(M,\xi)$ be a contact 3-manifold. Fix a Legendrian knot~$K$ viewed as an embedded submanifold.
    Then the map 
    \begin{eqnarray*}
        \cont_c(M,\xi) & \to & \lsub\!\left(K,(M, \xi)\right) \\
        \phi & \mapsto & \phi(K)
    \end{eqnarray*}
    is a locally trivial fibration with fibre~$\cont_c((M,\xi),K)$, where recall the latter is the space of contactomorphisms with compact support which fix~$K$ setwise.
\end{theorem}

\begin{proof}
    The parametrised Legendrian embedding space~$\lemb\!\left(K,(M, \xi)\right)$ is~$\cont_c(M, \xi)$-locally retractile, by \emph{Theorem \ref{thm:parametrised_fiber_bundle}}. \emph{Proposition \ref{prop:param_to_unparam}} gives that the map
    \[u \colon \lemb\!\left(K,(M, \xi)\right) \longrightarrow \lsub\!\left(K,(M, \xi)\right)\]
     is a locally trivial fibration; in particular it has local sections. It is also~$\cont_c(M, \xi)$-equivariant: for any~$\varphi \in \cont_c(M, \xi)$ and any parametrisation~$\gamma \colon S^1 \to (M, \xi)$ of~$K$, the embedded submanifold~$u\!\left(\varphi\!\left(\gamma\!\left(S^1\right)\right)\right) \subseteq (M, \xi)$ precisely coincides with the image under the action of~$\varphi$ on the embedded submanifold~$u\!\left(\gamma\!\left(S^1\right)\right)$. Hence, applying \emph{Lemma \ref{lem:properties of G-locally retractile spaces}, (\ref{lemitem:equivariant map with local sections implies target is locally retractile})}, we obtain that~$\lsub\!\left(K,(M, \xi)\right)$ is~$\cont_c(M, \xi)$-locally retractile. But now we can apply the Cerf-Palais criterion (\emph{Lemma \ref{lem:properties of G-locally retractile spaces}, (\ref{lemitem:Cerf-Palais criterion})}) to conclude that the map
    \[\cont_c(M,\xi)  \to  \lsub\!\left(K,(M, \xi)\right)\]
    is a locally trivial fibration, as required. 
    
    The stabilisers of the~$\cont_c(M, \xi)$-action on~$\lsub\!\left(K,(M, \xi)\right)$ are all isomorphic to the subgroup~$\cont_c((M, \xi), K)$ of contactomorphisms which fix~$K$ setwise.
\end{proof}

We are now ready to prove the first of our main theorems.

\begin{customthm}{\ref{thm:legendrian motion group as pi_1 of embedding space}}[\ref{customthm:legendrian motion group as pi_1 of embedding space}]
    The Legendrian motion group~$\lmotion{K}{(M, \xi)}$ is isomorphic to~$\pi_1\!\left(\mathcal{LE}(\iota)\right)$.
\end{customthm}

\begin{proof}
    The proof is an immediate application of the Five Lemma. In the following commutative diagram, the left column comes from the long-exact sequence for relative homotopy groups and the right column comes from the long-exact sequence associated to the fibration in \emph{Theorem \ref{thm:unparametrised_fiber_bundle}}.
    \[\begin{tikzcd}[cramped]
	\cdots & \cdots \\
	{\pi_1(\cont_c((M,\xi),K)),\iota)} & {\pi_1(\cont_c((M,\xi),K)),\iota)} \\
	{\pi_1(\cont_c(M,\xi))} & {\pi_1(\cont_c(M,\xi))} \\
	{\lmotion{K}{(M, \xi)}} & {\pi_1(\lsub\!\left(K,(M, \xi)\right))} \\
	{\pi_0(\cont_c((M,\xi),K)),\iota)} & {\pi_0(\cont_c((M,\xi),K)),\iota) } \\
	0 & 0
	\arrow[from=1-1, to=2-1]
	\arrow[from=1-2, to=2-2]
	\arrow["\cong"{description}, from=2-1, to=2-2]
	\arrow[from=2-1, to=3-1]
	\arrow[from=2-2, to=3-2]
	\arrow["\cong"{description}, from=3-1, to=3-2]
	\arrow[from=3-1, to=4-1]
	\arrow[from=3-2, to=4-2]
	\arrow[from=4-1, to=4-2]
	\arrow[from=4-1, to=5-1]
	\arrow[from=4-2, to=5-2]
	\arrow["{{\cong }}"{description}, from=5-1, to=5-2]
	\arrow["\cong"{description}, from=5-1, to=6-1]
	\arrow["\cong"{description}, from=5-2, to=6-2]
\end{tikzcd}\]

Note that with our conventions~$\pi_1(\lsub\!\left(K,(M, \xi)\right))$ really means~$\pi_1(\mathcal{LE}(\iota),\iota(K))$, since we are fixing a preferred basepoint (the inclusion) and working in that connected component.
\end{proof}

\subsection{Some basic properties of Legendrian motion groups}\label{subsec:basic properties of Legendrian motion groups}

We can now prove some properties of Legendrian motion groups.

\begin{proposition}\label{proposition:moving_around}
    If~$\varphi \in \cont_c((M,\xi))$, then~$\lmotion{K}{(M, \xi)} \cong \lmotion{\varphi(K)}{(M, \xi)}$. 
\end{proposition}
\begin{proof}
    The conjugation action~$\varphi' \mapsto \varphi \varphi' \varphi^{-1}$ on~$\cont_c(M, \xi)$ induces a homomorphism 
    \[\pi_1 \!\left( \cont_c(M,\xi), \cont_c((M,\xi),K), \id_M \right) \to \pi_1 \!\left( \cont_c(M,\xi), \cont_c((M,\xi),\varphi(K)), \id_{M} \right)\]
    which induces an isomorphism~$\lmotion{K}{(M, \xi)} \cong \lmotion{\varphi(K)}{(M, \xi)}$.
\end{proof}

\begin{corollary}[\ref{customcor:legendrian isotopy implies isomorphism}]\label{corollary:legendrian_isotopy_implies_isomorphism}
    If~$K,K'$ are two Legendrian isotopic links in~$(M,\xi)$, then their Legendrian motion groups are isomorphic.
\end{corollary}

\emph{Corollary \ref{corollary:legendrian_isotopy_implies_isomorphism}} can be viewed as a potential strategy to obstruct Legendrian isotopy between knots (and links) in contact 3-manifolds.

There is a natural forgetful map~$\lmotion{K}{(M,\xi)} \to \motion{K}{M}$. At the level of ($\pi_1$ of) embedding spaces, this map takes a loop of (unparametrised) Legendrian embeddings of~$K$ and sends it to the corresponding loop of smooth submanifolds. The following is immediate.

\begin{proposition}\label{prop:commutative_diagram}
    The following diagram is commutative.
    \[\begin{tikzcd}[cramped]
	{\lmotion{K}{(M, \xi)}} && {\motion{K}{M}} \\
	{\pi_1(\mathcal{LE}(\iota),\iota(K))} && {\pi_1(\mathcal{E}(\iota),\iota(K))}
	\arrow[from=1-1, to=1-3]
	\arrow["\cong", from=1-1, to=2-1]
	\arrow["\cong"', from=1-3, to=2-3]
	\arrow[from=2-1, to=2-3]
\end{tikzcd}\]
\end{proposition}

We have highlighted \emph{Proposition \ref{prop:commutative_diagram}} because it shows clearly how the notion of a Legendrian motion group could be helpful for understanding the space of Legendrian embeddings. In particular, one can potentially generate non-trivial loops of Legendrian embeddings by lifting trivial loops of smooth motions to non-trivial loops of Legendrian motions via the top row of the diagram.

\section{Motion groups and (contact) mapping class groups}\label{sec:monodromy maps}

\subsection{A map from Legendrian motions to contactomorphisms of the exterior}

This section characterises the Legendrian motion group of any Legendrian knot in the standard tight~$3$-sphere.

\begin{theorem}[\ref{customthm:legendrian motions as contactomorphisms of exterior}]\label{thm:legendrian motions as contact mapping classes}
    Let~$K$ be a Legendrian knot in~$\std$.
    \begin{enumerate}[(a)]
        \item If~$K$ is Legendrian non-invertible, then there is an isomorphism
        \[\Lambda \colon \lmotionstd{K} \xlongrightarrow{\cong} \cmcgext{K}.\]
        \item If instead~$K$ is Legendrian invertible, then there is an isomorphism
        \[\Lambda \colon \lmotionstd{K} \xlongrightarrow{\cong} \cmcgext{K} \rtimes \faktor{\mathbb{Z}}{2\mathbb{Z}}.\]
        \end{enumerate}
\end{theorem}

\begin{proof}
    Since we will always be working with the standard contact structure on~$S^3$, we will write~$\xi = \xi_{std}$, and suppress it from the notation where possible.

    Start with an embedded Legendrian knot~$K \subset S^3$, with a preferred parametrisation given by the inclusion~$\iota$. From \emph{Theorem \ref{thm:legendrian motion group as pi_1 of embedding space}}, we have an isomorphism
    \begin{eqnarray*}
        \lmotion{K}{S^3} & \xlongrightarrow{\cong} & \pi_1(\mathcal{LE}(\iota),\iota(K)) \\ 
        \left[\gamma\right] & \longmapsto & \left[\gamma(K)\right].
    \end{eqnarray*}
    Now, by \emph{Proposition \ref{prop:param_to_unparam}}, we have a locally trivial fibration
    \begin{align}
        \diff(S^1) \hookrightarrow \lemb\!\left(K,S^3\right) & \to \lsub\!\left(K,S^3\right) \label{eq:param_to_unparam}\\ 
        \varphi \; & \mapsto \; \varphi(K), \nonumber
    \end{align}
    where the fibre corresponds to reparametrisation: e.g., the fibre over~$\iota$ includes as~$\phi \mapsto \iota \circ \phi$. 

    Fix a point~$(p_0, v_0) \in \mathbb{S}(\xi)$ in the unit sphere bundle of~$\xi$. Let 
    \[\lemb_{(p_0, v_0)}\!\left(K,S^3\right) \coloneqq \left\{\gamma \in \lemb\!\left(K,S^3\right) \colon \!\left(\gamma(0), \gamma'(0)\right) = (p_0, v_0)\right\}\]
    denote the space of \emph{long Legendrian embeddings} of~$K$ into~$S^3$ realising the knot type~$K$ and with point `at infinity' of type~$(p_0, v_0)$. There is a homotopy equivalence \cite{Fernandez-Aguinaga-Presas} 
    \begin{align}\label{eqn:htpy equiv with long embeddings times U(2)}
        \lemb\!\left(K,S^3\right) \simeq \lemb_{(p_0,v_0)}\!\left(K,S^3\right)\times \operatorname{U}(2).
    \end{align}

    Here the Lie group~$\operatorname{U}(2)$ is precisely the sphere bundle~$\mathbb{S}(\xi)$, and it acts by contactomorphism on~$(S^3, \xi)$. The homotopy equivalence (\ref{eqn:htpy equiv with long embeddings times U(2)}) is given by the map
    \[\varphi \;\; \mapsto \;\; (A_{\varphi}^{-1}\cdot \varphi, A_{\varphi}),\]
    which has inverse
    \[A \cdot \psi \;\; \mapsfrom \;\; \!\left(\psi, A\right).\]

    Using the homotopy equivalence (\ref{eqn:htpy equiv with long embeddings times U(2)}) we obtain a fibration homotopy equivalent to (\ref{eq:param_to_unparam}):
    \begin{align}\label{eqn:fibration with long embeddings times U(2) as total space}
        \diff(S^1) \hookrightarrow \lemb_{(p_0,v_0)}\!\left(K,S^3\right)\times \operatorname{U}(2) & \xrightarrow{h} \lsub\!\left(K,S^3\right) \\ 
        \!\left(\psi, A\right) \; & \mapsto \; \!\left(A \cdot \psi\right)(K), \nonumber
    \end{align}
    where now the fibre over~$\iota$ includes as 
    \[g \colon \phi \mapsto \!\left(A_{\iota \circ \phi}^{-1} \cdot \!\left(\iota \circ \phi\right),\;A_{\iota \circ \phi}\right).\]

    Set~$(p_0, v_0) \coloneqq \!\left(\iota(0), \iota'(0)\right)$. Picking a basepoint~$\iota(K) \in \lsub\!\left(K,S^3\right)$, and picking the basepoint~$\!\left(\iota, \operatorname{id}\right)$ in the fibre~$h^{-1}\!\left(\iota(K)\right) \subseteq \lemb_{(p_0, v_0)}\!\left(K,S^3 \right) \times \operatorname{U}(2)$ to be its lift under~$h$, form the long exact sequence of homotopy groups associated to the fibration (\ref{eqn:fibration with long embeddings times U(2) as total space}). We obtain
    \[\begin{tikzcd}\tag{$\dagger$}\label{eq:long_exact_sequence_fibration}
	   \vdots && \\
	   {\pi_1\!\left(\diff\!\left(S^1\right), \operatorname{id}\right)} & {\pi_1\!\left(\lemb_{(p_0, v_0)}\!\left(K,S^3\right), \iota\right) \times \pi_1\!\left(\operatorname{U}(2), \operatorname{id}\right)} & {\pi_1\!\left(\mathcal{LE}(\iota), \iota(K)\right) } \\
	   && {0 \xleftarrow{g_\ast} \faktor{\mathbb{Z}}{2\mathbb{Z}},}
	   \arrow["{d_2}"', from=1-1, to=2-1]
	   \arrow["{g_1}", from=2-1, to=2-2]
	   \arrow["{h_1}", from=2-2, to=2-3]
	   \arrow["{d_1}", from=2-3, to=3-3]
    \end{tikzcd}\]
    where we have used functoriality of~$\pi_1$ in the middle term.

    \begin{claim*}
       ~$\im(g_1) \subseteq \pi_1\!\left(\U(2),\id\right)$.
    \end{claim*}

    \begin{proof}[Proof of Claim.]
        By definition of the action of~$\pi_1(\diff(S^1),\id)$ on~$\pi_1(\lemb_{(p_0,v_0)}\!\left(K,S^3\right),\iota)$, a loop of diffeomorphisms of~$S^1$ reparametrises the long knots. However, these are pointed knots, so in each of these diffeomorphisms, the image of~$0$ is always~$p_0$. Thus the loop of diffeomorphisms maps to a loop of reparametrisations of the interval. By the obvious `straightening', this is a null-homotopic loop of long Legendrian knots.
    \end{proof}

    Hence we can extract the following exact sequence:
    \[\begin{tikzcd}[cramped]\tag{$\ddagger$}\label{eq:extracted exact sequence}
	   0 & {\pi_1\!\left(\lemb_{(p_0, v_0)}\!\left(K, \std\right)\right)} & {\pi_1\!\left(\mathcal{LE}(\iota), \iota(K)\right)} & {\pi_0\!\left(\diff\!\left(S^1\right)\right)} & 0.
	   \arrow[from=1-1, to=1-2]
	   \arrow["{\overline{h}_1}", from=1-2, to=1-3]
	   \arrow["{d_1}", from=1-3, to=1-4]
	   \arrow[from=1-4, to=1-5]
    \end{tikzcd}\]
    
    We now unpack the sequence (\ref{eq:long_exact_sequence_fibration}). The connecting homomorphism~$d_1$ is defined as follows. A Legendrian motion~$f = \{f_t\} \in \lmotionstd{K}$ lifts to a path~$\{\widetilde{f}_t\}$ in~$\lemb\!\left(K, \std\right)$. The endpoint of this path can be written~$\widetilde{f}_1 = f_1 \circ \phi_1$, for~$\phi_1 \in \diff\!\left(S^1\right)$. Then~$d_1(f)$ is the class~$[\phi_1] \in \pi_0\!\left(\diff\!\left(S^1\right)\right) \cong \faktor{\mathbb{Z}}{2\mathbb{Z}}$. Hence if a motion corresponds to a Legendrian inversion of~$K$, then it maps to the generator of~$\pi_0\!\left(\diff\!\left(S^1\right)\right)$ under~$d_1$. 

    Thus if~$K$ is not Legendrian invertible, then~$\ker(d_1) = \pi_1\!\left(\mathcal{LE}(\iota), \iota(K)\right)$. Hence from (\ref{eq:extracted exact sequence}),~$\overline{h}_1$ is both injective and surjective, and we have an isomorphism
    \begin{equation}\label{eq:L motion group iso to pi1 of long embeddings}
        \pi_1\!\left(\lemb_{(p_0, v_0)}\!\left(K, \std\right)\right) \cong \pi_1\!\left(\mathcal{LE}(\iota), \iota(K)\right).
    \end{equation}

    If on the other hand~$K$ is Legendrian invertible, let~$\{f_t\}$ be a motion given by some isotopy which inverts~$K$, obtained by looking at the trace of submanifolds. Under~$d_1$ this maps to the generator of~$\pi_0\!\left(\diff\!\left(S^1\right)\right)$, so now by (\ref{eq:extracted exact sequence}) we have  that~$\pi_1\!\left(\mathcal{LE}(\iota), \iota(K)\right)$ is an extension of~$\pi_1\!\left(\lemb_{(p_0, v_0)}\!\left(K, \std\right)\right)$ by~$\faktor{\mathbb{Z}}{2\mathbb{Z}}$. Moreover, a splitting is given by mapping the generator of~$\faktor{\mathbb{Z}}{2\mathbb{Z}}$ to any motion realising an inversion of~$K$.

    Finally, in \cite{Min-Fernandez,Fernandez-Aguinaga-Presas} it is shown that there is a homotopy equivalence \begin{equation}\label{eq:homotopy_equivalence_loop_space_contactomorphism}
        \Omega\!\left(\lemb_{(p_0,v_0)}\!\left(K,S^3\right),\iota\right) \simeq \cont^{\partial}(C(K),\xi|)
    \end{equation} 
    where~$\Omega$ denotes the pointed loop space, and where~$\cont^{\partial}(C(K),\xi|)$ is the space of contactomorphisms of the Legendrian knot exterior (with contact structure induced by~$\xi$) that fix the boundary. Putting this together with the identification~$\lmotionstd{K} \cong \pi_1\!\left(\mathcal{LE}(\iota), \iota(K)\right)$ from \emph{Theorem \ref{thm:legendrian motion group as pi_1 of embedding space}} gives an isomorphism
    \[\cmcgext{K} \cong \lmotionstd{K},\]
    which is (a) from the statement.
    
    Similarly in the case of~$K$ invertible, we obtain
    \[\cmcgext{K}\rtimes \faktor{\mathbb{Z}}{2\mathbb{Z}} \cong \lmotionstd{K},\]
    which is (b).

\end{proof}

We now collect some remarks about this construction.

First of all, the isomorphisms in \emph{Theorem \ref{thm:legendrian motions as contact mapping classes}} can be described explicitly. We sketch here the map for (a); the map for (b) is the same but with an extra factor keeping track of possible inversions. The isomorphism in (\ref{eq:L motion group iso to pi1 of long embeddings}) says that we can view any loop~$f = \{f_t\}$ of unparametrised embeddings as a loop~$f'$ of parametrised embeddings such that~$f'_t\vert_{U_{p_0}} = \id$ for some open neighbourhood~$U_{p_0}$ of the point~$\iota^{-1}(p_0)$. All of these parametrised embeddings~$f'_t$ have the same fixed tangential direction~$v_0$ at~$p_0$. Next we put~$f'$ through the isomorphism 
\[\pi_1\!\left(\lemb_{(p_0,v_0)}\!\left(K,S^3\right),\iota\right) \xlongrightarrow{\cong} \cmcgext{K}\]
as follows. Writing~$f' \colon [0, 1] \to \lemb_{(p_0,v_0)}\!\left(K,S^3\right)$ such that~$f'_0 = f'_1 = \iota$, we can lift~$f'$ to a path~$\phi = \{\phi_t\}$ in~$\cont\!\left(S^3\right)$ such that~$\phi_0 = \id$,~$\phi_t \circ \iota = f'_t$, and~$\phi_t$ is fixed on some neighbourhood~$U_{p_0}$ for all~$t$. Since~$f'_1 = \iota$, the endpoint~$\phi_1$ preserves~$\iota(K)$ pointwise. This path can be homotoped to a path~$\widetilde{\phi}$ which fixes a tubular neighbourhood of the knot pointwise. Now restrict to obtain
\[\overline{\phi} \coloneqq \widetilde{\phi}_1\vert_{C(K)} \in \cont^\partial\!\left(C(K), \xi_\vert\right),\]
and the isomorphism 
\[\lmotionstd{K} \to \cmcgext{K}\]
is given by~$[f] \mapsto [\overline{\phi}]$.

\begin{remark}
    By (\ref{eq:homotopy_equivalence_loop_space_contactomorphism}), we can conclude that in order to produce non-trivial motions of~$K$ one can alternatively study the space of contactomorphisms of the knot exterior~$C(K)$. This is still a very tough task, but with the recent spark in interest towards contactomorphisms groups of compact 3-manifolds we find this reformulation quite tantalising.
\end{remark}

We finish with an interesting interpretation of \emph{Theorem \ref{thm:legendrian motions as contact mapping classes}}.
K\'alm\'an \cite{Kalman} introduced a series of examples of loops of Legendrian knots that are non-contractible contact-wise, yet contractible in~$\emb(S^1,\mathbb{R}^3)$ (or in~$\emb(S^1,S^3)$). In particular, he produces loops of \emph{max-tb} Legendrian positive~$(p,q)$-torus knots~$T_{p,q}$ with this property, where \emph{max-tb} means they have the maximal possible Thurston--Bennequin number. To prove the non-triviality of such loops K\'alm\'an uses a monodromy map taking values in the Legendrian Contact Homology (LCH) of~$T_{p,q}$ \cite{Kalman}. Alternatively, Fern\'andez--Mart\'inez-Aguinaga--Presas prove that such loops are non-trivial as loops of so-called \emph{formal} Legendrian embeddings and as loops of unparametrised knots \cite{Fernandez-Aguinaga-Presas}.

Such a loop of Legendrian submanifolds of~$(S^3,\xi_{std})$ should correspond, by \emph{Theorem \ref{thm:legendrian motion group as pi_1 of embedding space}}, to an element of the Legendrian motion group. The following language in which to do this is found in \cite[1.1.2]{Min-Fernandez}. Let~$U$ be the max-tb Legendrian unknot in~$(S^3,\xi_{std})$. We can take a standard tubular neighbourhood~$N_U$ \cite{Etnyre_lecture_notes} and foliate its boundary~$\partial N_U$ by ruling curves of slope~$q/p$, where~$p$,~$q$ are coprime positive integers. Such curves represent max-tb Legendrian positive~$(p,q)$-torus knots. Now we can obtain the K\'alm\'an loop by rotating the solid torus~$N_U$ around about its core by a~$2\pi/p$ angle. This extends to a loop of contactomorphisms of~$(S^3,\xi_{std})$, so by \emph{Theorem \ref{thm:legendrian motions as contact mapping classes}} we have the following.

\begin{proposition}
     The K\'alm\'an loop is a nontrivial element of~$\lmotion{T_{p,q}}{S^3}$.
\end{proposition}
This is mentioned in different terms in \cite{Min-Fernandez}, where it is shown via Lemma 3.5 of that paper that the K\'alm\'an loop is actually generated by a \emph{contact torus twist}.

\appendix

\section{Smooth mapping classes as motions, and the Dahm homomorphism}\label{appendix:smooth_case}

In the thesis introducing smooth motion groups \cite{Dahm62}, Dahm also defined a map 
\[D \colon \motion{N}{M} \to \aut\!\left(\pi_1\!\left( M \setminus N \right)\right),\]
which encodes the natural action of a motion on the complement~$M \setminus N$. Given that we view \emph{Theorem \ref{thm:legendrian motions as contact mapping classes}} as a moral contact substitute for the Dahm homomorphism in the contact case, it is natural to wonder whether one can reinterpret the smooth Dahm homomorphism via a construction mimicking the one laid out in \S\ref{sec:monodromy maps}.

We focus here on the case~$M=S^3$ and a smooth knot~$K$. As before, there is a distinction between knots that are (smoothly) invertible and those that are not. The fibre of the naturally arising fibration is~$\diff(S^1)$ or~$\diff^+(S^1)$ respectively; we confine ourselves here to the (slightly harder) invertible case.

Start from the isomorphism
\begin{eqnarray*}
    \motion{K}{S^3} & \xlongrightarrow{\cong} & \pi_1\!\left(\calE(\iota), \iota(K)\right) \\
    \left[f\right] & \longmapsto & \left[f(K)\right]
\end{eqnarray*}

We have a locally trivial fibration 
\begin{eqnarray*}\label{eq:fibration_smooth}
    \diff(S^1) \hookrightarrow \emb(K,S^3) & \to & \sub(K,S^3) \\
    \varphi & \mapsto \varphi(L),
\end{eqnarray*}
as well as a fibration \cite{Budney}
\begin{equation}
    \emb_{(p_0,v_0)}(K,S^3) \to \emb(K,S^3) \to \faktor{\SO(4)}{\SO(2)},
\end{equation}
where~$\emb_{p_0,v_0}(S^3,K)$ is the space of \emph{long smooth knots} of smooth knot type~$K$.
We know more, in fact: there is a homotopy equivalence \cite{Budney-Cohen}
\begin{equation}\label{eq:budney-cohen}
    \SO(4) \times_{\SO(2)}  \emb_{(p_0,v_0)}\!\left(K,S^3\right) \simeq \emb(K,S^3)
\end{equation}
where the first is the \emph{balanced product} of the two~$\SO(2)$-spaces. The~$\SO(2)$-action on~$\SO(4)$ is the natural right action of the matrix group, and the left~$\SO(2)$-action on~$\emb_{(p_0,v_0)}(S^3,K)$ is by post-composition with a rotation along the compactified long-axis. Moreover, we know from \cite{Budney} that \[\diff^{\partial}\!\left(C(K)\right) \cong \Omega \!\left(\emb_{(p_0,v_0)}(K,S^3)\right),\]
where here~$\diff^{\partial}\!\left(C(K)\right)$ is the group of diffeomorphisms of the knot exterior~$C(K)$ fixing the boundary pointwise. In particular 
\begin{equation}\label{eq:smooth MCG iso to pi1 of long embeddings}
    \pi_0\diff^{\partial}\!\left(C(K)\right) \cong \pi_1\!\left( \emb_{(p_0,v_0)}(K,S^3)\right).
\end{equation}

Combining the long exact sequence of homotopy groups associated to the fibration (\ref{eq:homotopy_equivalence_loop_space_contactomorphism}) with the homotopy equivalence (\ref{eq:budney-cohen}) yields the following sequence.
\[
\begin{tikzcd}[cramped]
	\cdots & {\pi_1\!\left(\diff\!\left(S^1\right), \id\right)} & {\pi_1\!\left(\SO(4) \times_{\SO(2)}  \emb_{(p_0,v_0)}\!\left(K,S^3\right)\right)} & {\pi_1\!\left(\sub\!\left(K, S^3\right), \iota(K)\right)} \\
	&&& {0 \longleftarrow \pi_0\!\left(\diff\!\left(S^1\right), \id\right)}
	\arrow[from=1-1, to=1-2]
	\arrow[from=1-2, to=1-3]
	\arrow[from=1-3, to=1-4]
	\arrow[from=1-4, to=2-4]
\end{tikzcd}
\]

The first map on the left is the inclusion of a longitudinal rotation -- this is the \emph{Gramain loop}. The~$\pi_1(\diff(S^1),\id)$ action by rotation \emph{along the knot} is transversal to the one given by the~$\SO(2)$ factor, where the rotation is meridional \emph{along the long axis}. 

The balanced product~$\SO(4)\times_{\SO(2)}\emb_{(p_0,v_0)}(K,S^3)$ is equivalent to the orbit space 
\[\frac{\SO(4) \times \emb_{(p_0,v_0)}(K,S^3)}{\SO(2)},\]
where the action is diagonal and on the left. Performing this replacement in the above long exact sequence, and recalling that~$\pi_1\!\left(\sub\!\left(K, S^3\right), \iota(K)\right) \cong \motion{K}{S^3}$, we obtain a map~$p_{mon}$ which fits into the following exact sequence:
\[\begin{tikzcd}
	\cdots & {\mathbb{Z}} & {\pi_1\!\left(\frac{\SO(4) \times \emb_{(p_0,v_0)}(K,S^3)}{\SO(2)}, (\id,\iota)\right)} & {\motion{K}{S^3}} & {\faktor{\mathbb{Z}}{2\mathbb{Z}}} & 0
	\arrow[from=1-1, to=1-2]
	\arrow[from=1-2, to=1-3]
	\arrow["{p_{mon}}", from=1-3, to=1-4]
	\arrow[from=1-4, to=1-5]
	\arrow[from=1-5, to=1-6]
\end{tikzcd}\]

By \cite{Budney-Cohen} we know that there is a fibre sequence
\[\emb_{(p_0,v_0)}\!\left(K,S^3\right) \xhookrightarrow{i} \frac{\SO(4) \times \emb_{(p_0,v_0)}\!\left(K,S^3\right)}{\SO(2)} \to \faktor{\SO(4)}{\SO(2)}\]
from which we obtain a map 
\[i_*\colon \pi_1\!\left(\emb_{(p_0,v_0)}\!\left(K,S^3\right), \iota \right) \longrightarrow \pi_1\!\left(\frac{\SO(4) \times \emb_{(p_0,v_0)}\!\left(K,S^3\right)}{\SO(2)}, (\id,\iota)\right).\]
Since~$\faktor{\SO(4)}{\SO(2)}$ is simply-connected,~$i_*$ is surjective. Note however that in general~$i_*$ needn't be injective.

We can fit everything together in the following diagram, where the rightmost map is the Dahm homomorphism.
\[\begin{tikzcd}[cramped]
	{ \pi_1\!\left(\emb_{(p_0,v_0)}\!\left(K,S^3\right), \iota \right) } & {\pi_1\!\left(\frac{\SO(4) \times \emb_{(p_0,v_0)}\!\left(K,S^3\right)}{\SO(2)}, (\id,\iota)\right)} & {\motion{K}{S^3}} & {\Out\!\left(\pi_1(C(K))\right)} \\
	\\
	{\pi_0\!\left(\diff^{\partial}\!\left(C(K)\right)\right)}
	\arrow["{{i_*}}", two heads, from=1-1, to=1-2]
	\arrow["{{p_{mon}}}", from=1-2, to=1-3]
	\arrow["{{D}}", from=1-3, to=1-4]
	\arrow["\cong", from=3-1, to=1-1]
	\arrow[bend right=10, from=3-1, to=1-4]
\end{tikzcd}\]

It would be tempting to say that the induced map
\[D \circ p_{mon} \circ i_\ast \colon \pi_0\!\left(\diff^\partial\!\left(C(K)\right)\right) \longrightarrow \Out\!\left(\pi_1\!\left(C(K)\right)\right)\]
is a factorisation of the Waldhausen map \cite[Corollary 7.5]{Waldhausen1968} through the Dahm homomorphism.
However, Waldhausen's map is always surjective. The Dahm homomorphism is, on the other hand, not always surjective -- see for example certain torus knots \cite{GoldsmithTorusLinksin3Sphere82}.

We finish with a corollary of the above discussion, applied to the case that~$K$ is non-invertible.

\begin{corollary}\label{cor:smooth noninvertible K all motions generated by complement diffeomorphisms}
    Let~$K$ be a smooth non-invertible knot in~$S^3$. Then~$p_{mon} \circ i_\ast$ is surjective. In particular, every nontrivial motion is generated by a diffeomorphism of the knot complement.
\end{corollary}

\footnotesize
\bibliographystyle{alpha}

\bibliography{Bibliography.bib}\medskip

\noindent Gabriel Corrigan \\
University of Glasgow \\
\href{mailto:g.corrigan.1@research.gla.ac.uk}
{g.corrigan.1@research.gla.ac.uk}

\noindent Isacco Nonino \\
University of Glasgow \\
\href{mailto:i.nonino.1@research.gla.ac.uk}
{i.nonino@research.gla.ac.uk}

\end{document}

%% file: running_preamble.tex
\usepackage{subfiles}
\usepackage{comment}
\usepackage{float}
\usepackage{marginnote}
\usepackage{tabu}
\usepackage{euscript}
\usepackage[dvipsnames]{xcolor}   	
\usepackage{quiver}
\usepackage{graphicx}			
\usepackage{amssymb}
\usepackage{mathrsfs}
\usepackage{amsthm}
\usepackage{amsmath}
\usepackage{stmaryrd}
\usepackage{tikz}
\usepackage{tikz-cd}
\usepackage{accents}
\usepackage{upgreek}
\usepackage{enumerate}
\usepackage{bm}
\usepackage{mathtools}
\usepackage{appendix}
\usepackage[all]{xy}
\usepackage{caption}
\usepackage{url}
\usepackage{float}
\usepackage{todonotes} 
\usepackage{colonequals}
\usepackage{bbm}
\usepackage{longtable}
\usepackage[full]{textcomp}
\usepackage[cal=cm]{mathalfa}
\usepackage{xparse}
\usepackage{comment}
\usepackage[noadjust]{cite} 

\usepackage{faktor} 
\usepackage{xfrac} 

\usepackage{derivative}

\usetikzlibrary{calc}
\usetikzlibrary{fadings}
\usetikzlibrary{decorations.pathmorphing}
\usetikzlibrary{decorations.pathreplacing}
\usepackage{tikz,tikz-cd,tikz-3dplot}
\usepackage{pgfplots}
\usetikzlibrary{arrows,shadows,positioning, calc, decorations.markings, 
hobby,quotes,angles,decorations.pathreplacing,intersections,shapes}
\usepgflibrary{shapes.geometric}
\usetikzlibrary{fillbetween,backgrounds}
\usetikzlibrary{arrows.meta} 

\usepackage[margin=1.05in]{geometry}

\tikzset{
  commutative diagrams/.cd, 
  arrow style=tikz, 
  diagrams={>=stealth}
}
\tikzset{
  arrow/.pic={\path[tips,every arrow/.try,->,>=#1] (0,0) -- +(0,4pt);},
  pics/arrow/.default={triangle 90}
}
\tikzset{->-/.style={decoration={
  markings,
  mark=at position .6 with {\arrow{latex}}},postaction={decorate}}
  }
\tikzset{
  c/.style={every coordinate/.try}
}

\theoremstyle{definition}
\newtheorem{innercustomthm}{Theorem}
\newenvironment{customthm}[1]
  {\renewcommand\theinnercustomthm{#1}\innercustomthm}
  {\endinnercustomthm}

\theoremstyle{definition}
\newtheorem{innercustomcor}{Corollary}
\newenvironment{customcor}[1]
  {\renewcommand\theinnercustomcor{#1}\innercustomcor}
  {\endinnercustomcor}

  \theoremstyle{definition}
\newtheorem{innercustomprop}{Proposition}
\newenvironment{customprop}[1]
  {\renewcommand\theinnercustomprop{#1}\innercustomprop}
  {\endinnercustomprop}
  
\makeatletter
\def\@tocline#1#2#3#4#5#6#7{\relax
  \ifnum #1>\c@tocdepth 
  \else
    \par \addpenalty\@secpenalty\addvspace{#2}%
    \begingroup \hyphenpenalty\@M
    \@ifempty{#4}{%
      \@tempdima\csname r@tocindent\number#1\endcsname\relax
    }{%
      \@tempdima#4\relax
    }%
    \parindent\z@ \leftskip#3\relax \advance\leftskip\@tempdima\relax
    \rightskip\@pnumwidth plus4em \parfillskip-\@pnumwidth
    #5\leavevmode\hskip-\@tempdima
      \ifcase #1
       \or\or \hskip 1em \or \hskip 2em \else \hskip 3em \fi%
      #6\nobreak\relax
    \dotfill\hbox to\@pnumwidth{\@tocpagenum{#7}}\par
    \nobreak
    \endgroup
  \fi}
\makeatother

\newcounter{marginnote}
\DeclareMathAlphabet{\mathpzc}{OT1}{pzc}{m}{it}

\usepackage[backref=page]{hyperref}
\hypersetup{
  colorlinks   = true,          
  urlcolor     = blue!65!black,          
  linkcolor    = blue!65!black,          
  citecolor   = blue             
}
\usepackage[capitalize,noabbrev]{cleveref}

\pgfplotsset{compat=1.18}

\theoremstyle{definition}
\newtheorem{theorem}{Theorem}[section]

\newtheorem*{claim*}{Claim}

\newtheorem{corollary}[theorem]{Corollary}

\newtheorem{lemma}[theorem]{Lemma}
\newtheorem{proposition}[theorem]{Proposition}
\newtheorem{remark}[theorem]{Remark}

\newtheorem*{runningexample*}{Running example}

\newtheorem*{aside*}{Aside}

\newtheorem{definition}[theorem]{Definition}

\newtheorem*{notation*}{Notation} 
\newtheorem{proposition-definition}[theorem]{Proposition-Definition}
\newtheorem{theorem-definition}[theorem]{Theorem-Definition} 

\DeclareMathOperator{\id}{id}

\newcommand{\Exp}{\operatorname{Exp}}

\newcommand{\bcd}{\begin{center}\begin{tikzcd}}
\newcommand{\ecd}{\end{tikzcd}\end{center}}

\newcommand{\calE}{\mathcal{E}}

\newcommand{\im}{\text{im }}

\newcommand{\motion}[2]{\mathcal{M}\!\left({#1}, {#2}\right)} 

\newcommand{\aut}{\operatorname{Aut}(A_\Gamma)} 

\newcommand{\diff}{\operatorname{Diff}} 
\newcommand{\cont}{\operatorname{Cont}} 
\newcommand{\emb}{\operatorname{Emb}} 
\newcommand{\sub}{\operatorname{Sub}} 
\newcommand{\lemb}{\operatorname{LEmb}}
\newcommand{\lsub}{\operatorname{LSub}} 
\newcommand{\lmotion}[2]{\ensuremath{\mathcal{L}\!\left(#1, #2\right)}}  
\newcommand{\lmotionstd}[1]{\ensuremath{\mathcal{L}\!\left(#1, \std\right)}}  
\newcommand{\cmcgext}[1]{\ensuremath{\pi_0\!\left(\cont^\partial\!\left(C(#1), \xi_\vert\right)\right)}}   
\newcommand{\std}{\ensuremath{\!\left(S^3, \xi_{std}\right)}}   
\newcommand{\tb}{\operatorname{tb}}
\newcommand{\rot}{\operatorname{rot}}

\DeclareMathOperator{\U}{U}
\DeclareMathOperator{\SO}{SO}
\DeclareMathOperator{\Out}{Out}

\newtheorem*{theorem:main}{Theorem~\ref{thm:legendrian motion group as pi_1 of embedding space}}
\newtheorem*{theorem:map_on_legendrian_motions}{Theorem~\ref{thm:map_on_legendrian_motions}}
\newtheorem*{theorem:waldhausen_map}
{Theorem~\ref{thm:factoring_waldhausen_map}}